\documentclass[a4paper,12pt]{article}
\usepackage{times, url}
\usepackage{graphicx}
\usepackage{subfigure}
\makeatletter
\def\Ddots{\mathinner{\mkern1mu\raise\p@
\vbox{\kern7\p@\hbox{.}}\mkern2mu
\raise4\p@\hbox{.}\mkern2mu\raise7\p@\hbox{.}\mkern1mu}}
\makeatother

\usepackage{amsmath,amssymb,amsthm,amsfonts}

\usepackage{graphicx,caption}

\newtheorem{theorem}{Theorem}[section]
\newtheorem{corollary}{Corollary}[theorem]

\begin{document}
\setcounter{page}{1} 

\vspace{10mm}

\begin{center}
{\LARGE \bf Recursive eigen extrusion : Expanding eigenbasis conjecture}
\vspace{8mm}

{\Large \bf M. Hariprasad}
\vspace{3mm}

Department of Computational and Data Sciences  \\ 
Indian Institute of Science, Bangalore-560012, India \\
e-mail: \url{mhariprasadkansur@gmail.com}
\end{center}
\vspace{10mm}

\noindent
{\bf Abstract:}
Consider $n$ linearly independent vectors in $\mathbb{C}^n$ which form columns 
of a matrix $A$. The recursive evaluation of eigen directions (normalized 
eigenvectors) of $A$ is the solution of an eigenvalue problem of the form 
$A_iX_i=X_i\Lambda_i$ with $i=0,1,2 \dots$; and here $\Lambda_i$ is the diagonal 
matrix of eigenvalues and columns of $X_i$ are the eigenvectors.
Note that $A_{i+1}=\phi(X_i)$ where $\phi$ normalizes 
all eigenvectors to unit $\mathcal{L}_2$ norm such that all diagonal elements 
$[\phi(X)^\dagger\phi(X)]_{jj}=1$.  It is to be 
proven that for any matrix $A_o$ and $n \leq 7$, the limiting set of matrices 
$A_i$ with $i \to \infty$ is the set of unitary matrices $U(n)$ with 
$X_i^\dagger X_i \to I$. Interestingly, this problem also represents a 
recursive map that maximizes some average distance among a set of $n$ points on the 
unit $n$-sphere. We first formally pose this conjecture, present extensive numerical 
results highlighting it, and prove it for special cases. 
\vspace{7mm}

\section{Introduction}
Geometrical properties of the unit $n$-sphere $\mathbb{U}^n$ ( 
\cite{akemann1970geometry},\cite{buser2010geometry}, 
\cite{bohnenblust1955geometrical} ) 
and distribution of points in its surface or volume 
(\cite{cui1997equidistribution}, \cite{katanforoush2003distributing}) 
are widely studied from theoretical and application perspective. Also, the 
structure of eigenvector matrices 
(\cite{trefethen1997numerical},\cite{hogben2006handbook}), 
its condition number (\cite{bauer1960norms}, \cite{chu1986generalization}) and 
correlation between 
its columns \cite{oja1985stochastic} are very well studied in the literature. In 
this article we study the limiting behavior of recursive evaluation of a linear 
operator and the resulting eigen directions, and more importantly its 
significant properties as a function of the dimension of space $n$. We 
conjecture that the limiting set of vectors are orthogonal only for $n \leq 7$, 
and this recursion is also equivalent to an iterative map maximizing the 
distance between $n$ points on a $n$-sphere. Along with symmetry and other 
interesting geometric properties, we note that the relative $n$-area of a sphere 
compared to a face of the larger circumscribing $n$-cube increases until $n$ increases to 
7, and decreases thereafter. Here we present a conjecture supported by numerical 
evidence, where this significant dimension $n$ =7, reappears. Note that a set of 
$n$ unit vectors have $\binom{n}{2}$ degrees of freedom given by the angles 
between them. Given this procedure of extrusion of eigen directions, the 
eigenvalue problem includes only the real and imaginary parts of the eigenvalue 
(magnitude and phase), and the norm, as the three parameters for every 
eigenvector. It should be noted that for $n\leq7$, $\binom{n}{2} \leq 3n$.

\section{Numerical observations and a conjecture}
In this section first we state the conjecture, and present the numerical 
observations. We say a property holds for almost all matrices, if for any matrix $A$ we have
matrix $A'$ such that $|A(i,j)-A'(i,j)| < \epsilon$ for any $\epsilon >0$ satisfying that property.
\newtheorem{Conjecture}{conjecture}
\begin{Conjecture}{\textbf {(Expanding Eigen basis) :}}\label{eb}
 For almost all matrices up to dimension 7, If we repeatedly construct the eigenvector matrices $X_i$, 
 each of them having columns of unit $\mathbb{L}_2$ norm, then in the limit of 
large $i$, we have $X_i^\dagger X_i \to I$.
\end{Conjecture}

\begin{Conjecture}{\textbf {(Expanding orthogonal similarity variant ) :}}\label{es}
  For almost all matrices, 
 up to dimension 7, If we repeatedly do the procedure of constructing the 
eigenvector matrix and similarity transformation
 with a random orthogonal (or unitary) matrix, $Y_{i} = Q_iZ_{i}Q_i^{\dagger}$, ($Z_1$ being 
eigenvector matrix of $X_1$,
 $Z_i$ being the eigenvector matrix of $Y_{i-1}$), then in the 
 limit of large $i$, we have $Y_i^\dagger Y_i \to I$.
\end{Conjecture}

\begin{Conjecture}{\textbf {(Expanding orthogonal product variant ) :}}\label{eq}
  For almost all matrices, 
 up to dimension 7, If we repeatedly do the procedure of constructing the 
eigenvector matrix and multiply it
 with a random orthogonal (or unitary)  matrix, $Y_{i} = Q_iZ_{i}$, ($Z_1$ being eigenvector 
matrix of $X_1$,
 $Z_i$ being the eigenvector matrix of $Y_{i-1}$), then in the 
 limit of large $i$, we have $Y_i^\dagger Y_i \to I$.\\
\end{Conjecture}

Here, convergence of $X_i^{\dagger}X_i$ need not necessarily imply convergence of $X_i$ itself to a particular 
unitary matrix. Instead, $X_i$ tend to be more unitary with increase in $i$.
Also the conjecture is stated for almost all matrices, because when algebraic and geometric multiplicities of an eigenvalue
are different, we get a rank deficient eigenvector matrix, which can never be orthogonal.
So we exclude all matrices which give rank deficient eigenvector matrix at any iteration.
Also, there can be cases when the eigenvector matrices can get stuck into loops as shown in section \ref{loop}. 

In figure \ref{Ur} and figure \ref{Gr}, uniform and Gaussian random 
matrices are considered. 
The following MATLAB commands are executed for verifying expanding eigenbasis conjecture,
\begin{align*}
 &[u ,v ] = eig (A); \\
 &for \text{	}i = 1 : N \\
 &\text{	}	[u, v] = eig(u);
\end{align*}
Determinant of $u^\dagger u$ is plotted 
over iteration from one to $N$. This is repeated for many matrices in a single plot. \\

In figure \ref{Cr},
we plot the average (over 
many matrices for a particular iteration) 
value of $\log(-\log(\det (u ^{\dagger} u)))$ with respect to the
iteration number for expanding eigenbasis and its variant conjectures.
\begin{itemize}
\item  It can be seen that at the $k^{\text{th}}$ iteration the 
$\det(u^{\dagger}u) \approx e^{-e^{-\frac{k}{2^d}}}$ where $d+2$ is dimension of 
the 
matrix, for expanding eigen basis conjecture (figure \ref{Cr1}).
\item Convergence for low dimensions (dimension two and three) have been slowed 
down in expanding eigenvector similarity 
variant (figure \ref{Cr2}). 
\item Convergence is faster (up to dimension six, we have $\det(u^{\dagger}u) 
\approx e^{-e^{-\frac{k}{2^{d-1}}}}$ , here $d+2$ representing dimension of the 
matrix,) for expanding orthogonal product variant where eigenvector 
matrix is multiplied by an orthogonal matrix before constructing next 
eigenvector matrix (figure \ref{Cr3}).
\end{itemize}

\begin{figure}
\centering     
\subfigure[Convergence: 100 matrices are considered for each dimension]{\label{fig:a}\includegraphics[width=70mm]{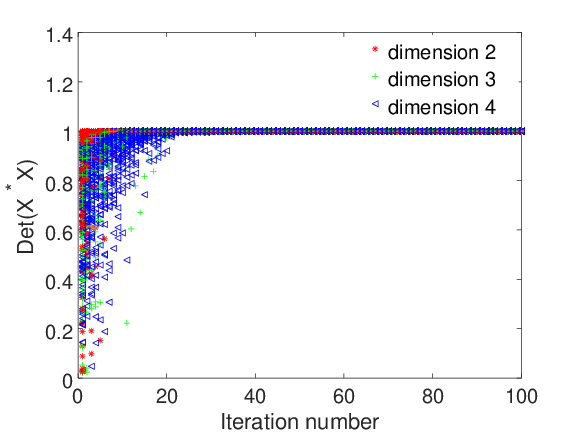}}
\subfigure[Convergence: 100 matrices are considered for each dimension]{\label{fig:b}\includegraphics[width=70mm]{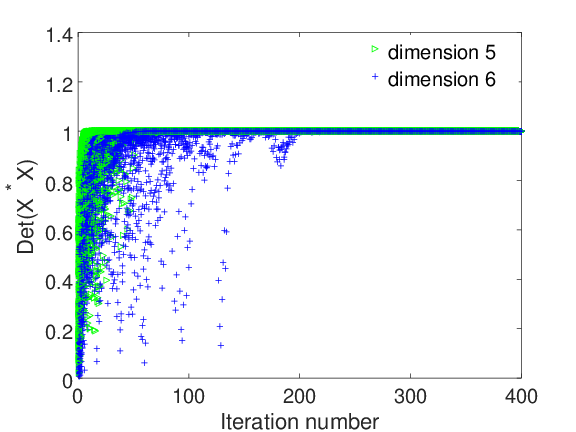}}
\subfigure[Convergence: 10 matrices are considered for each dimension]{\label{fig:c}\includegraphics[width=70mm]{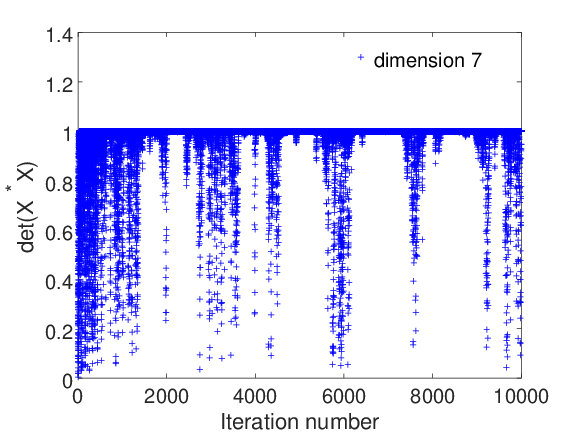}}
\subfigure[Convergence: 10 matrices are considered for each dimension]{\label{fig:d}\includegraphics[width=70mm]{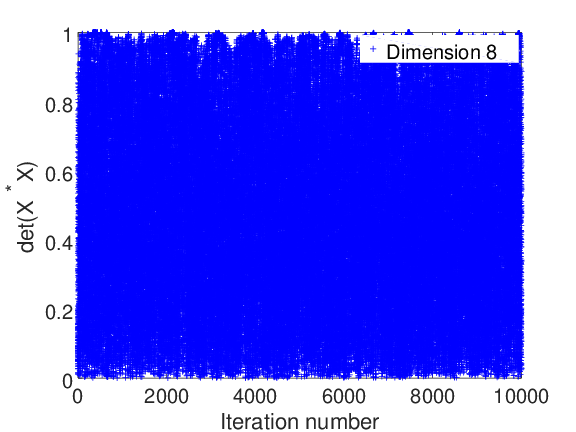}}
\caption{Uniform $\mathcal{U}(0,1)$ random matrices used for verification of the 
conjecture. 
Convergence of $\det(X_i^{\dagger}X_i)$ is observed for dimension less than 7, 
instability is observed at $n$ = 7, and no convergence is observed for a 
dimension $n$ = 8.}\label{Ur}
\end{figure}

\begin{figure}
\centering     
\subfigure[100 matrices are considered for each dimension]{\label{fig:a}\includegraphics[width=70mm]{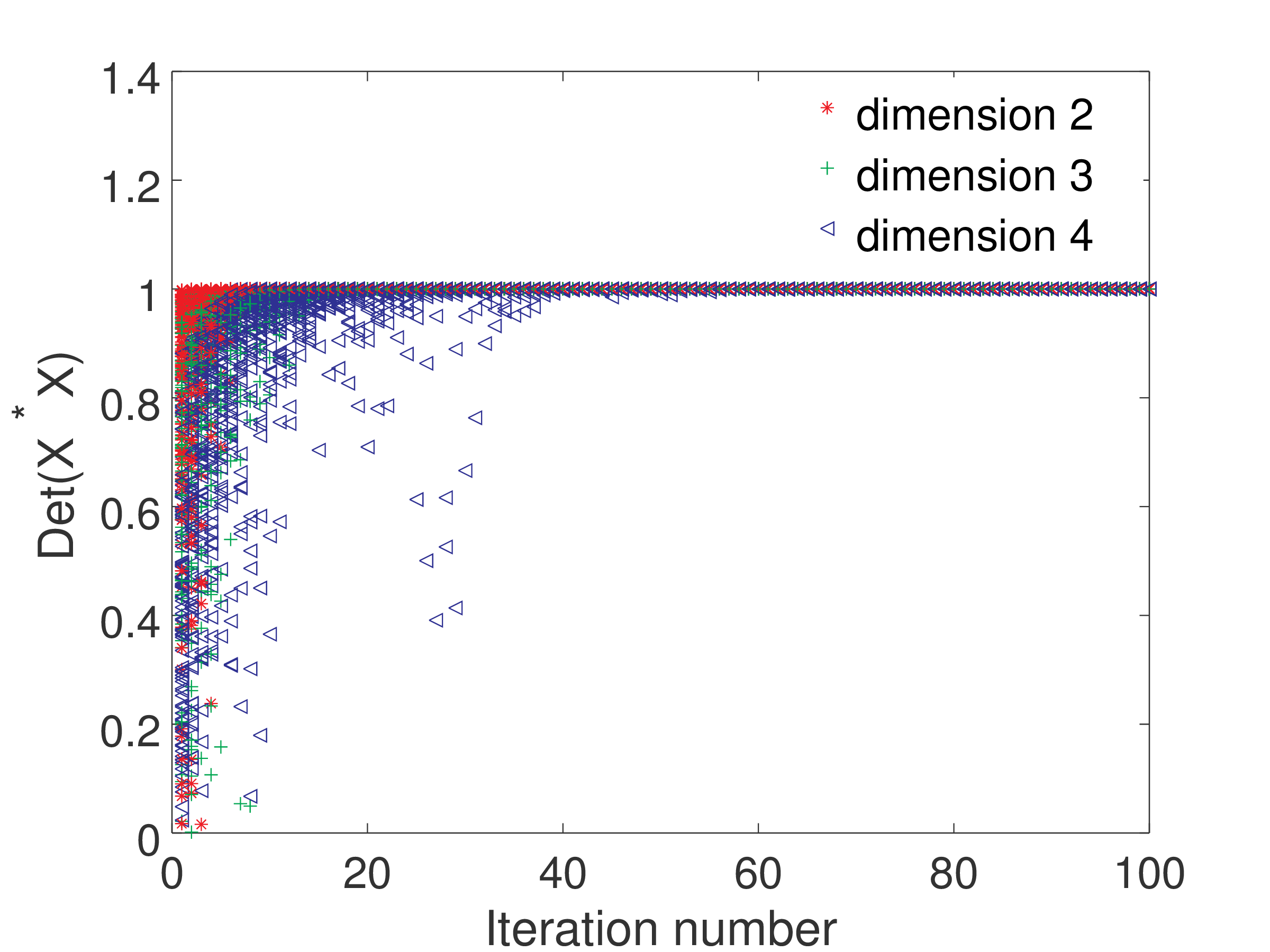}}
\subfigure[100 matrices are considered for each dimension]{\label{fig:b}\includegraphics[width=70mm]{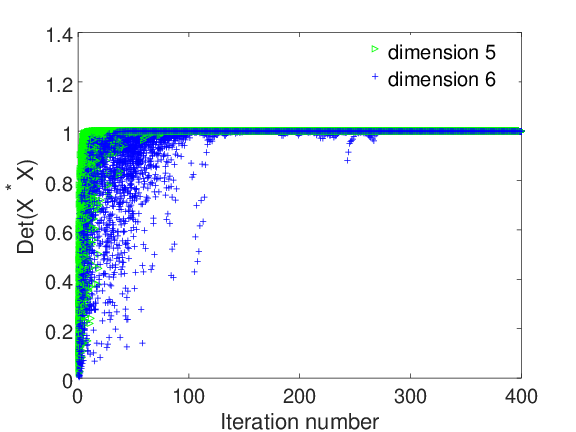}}
\subfigure[Convergence: 10 matrices are considered for each dimension]{\label{fig:c}\includegraphics[width=70mm]{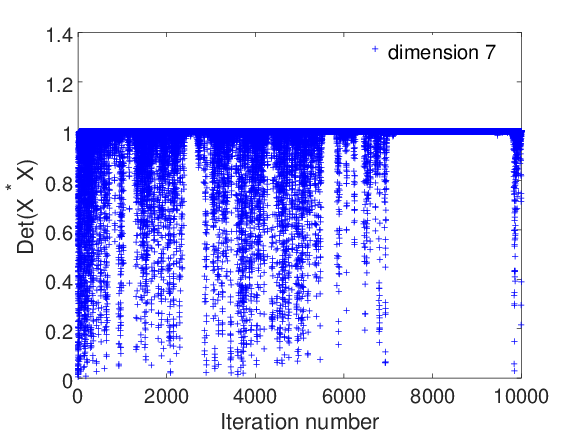}}
\subfigure[Convergence: 10 matrices are considered for each dimension]{\label{fig:d}\includegraphics[width=70mm]{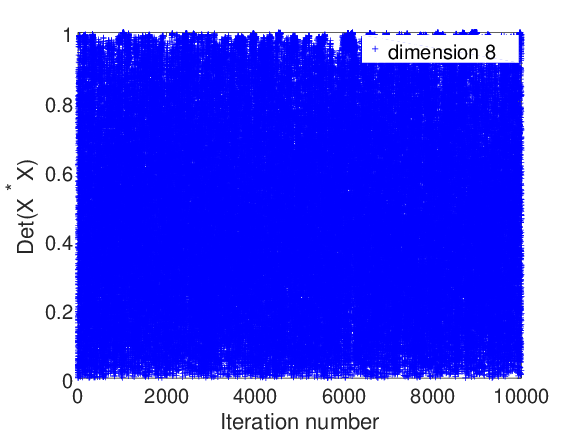}}
\caption{Convergence:  Gaussian $\mathcal{N}(0,1)$ are used random matrices for verification 
of the conjecture. Convergence is observed for dimensions less than 7, 
instability is observed at $n$ = 7, and no convergence is observed for a 
dimension $n$ = 8.}\label{Gr}
\end{figure}

\begin{figure}
\subfigure[Conjecture 1 : expanding eigenbasis]
{\includegraphics[width = 7.5 cm]{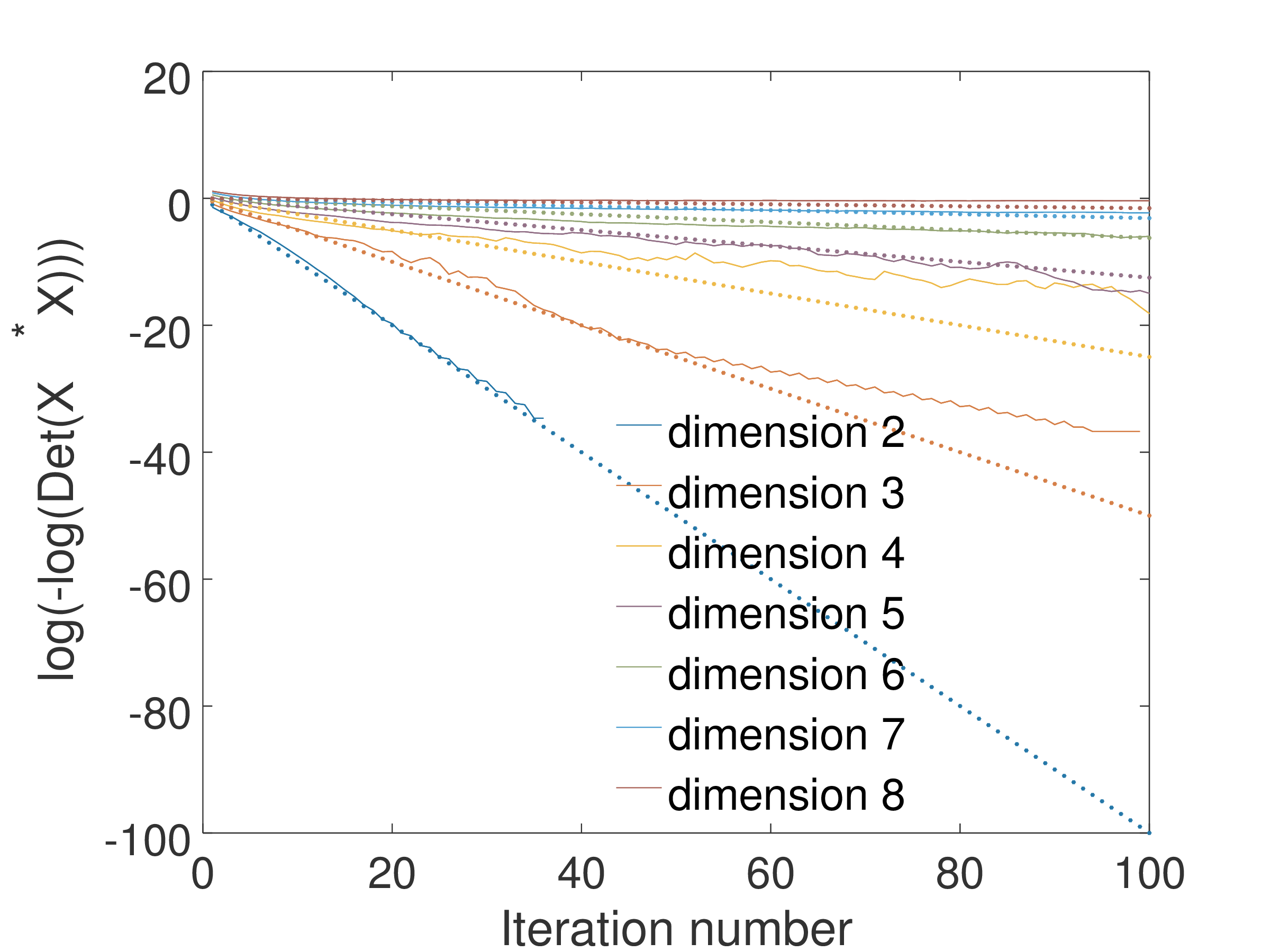}\label{Cr1}}\quad
\subfigure[Conjecture 2 :  orthogonal 
similarity variant]
{ \includegraphics[width = 7.5 cm]{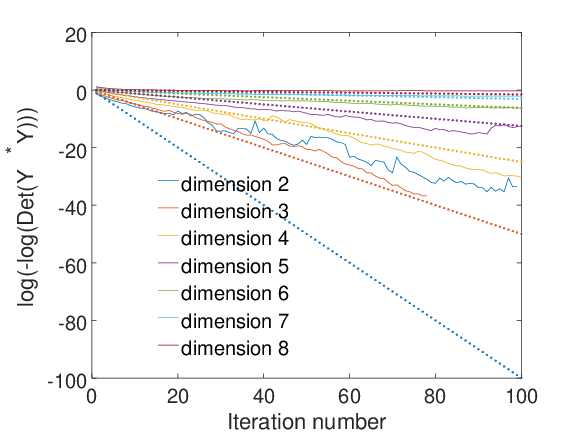}\label{Cr2}}
\centering
\subfigure[Conjecture 3 : orthogonal product variant]
{ \includegraphics[width = 7.5 cm]{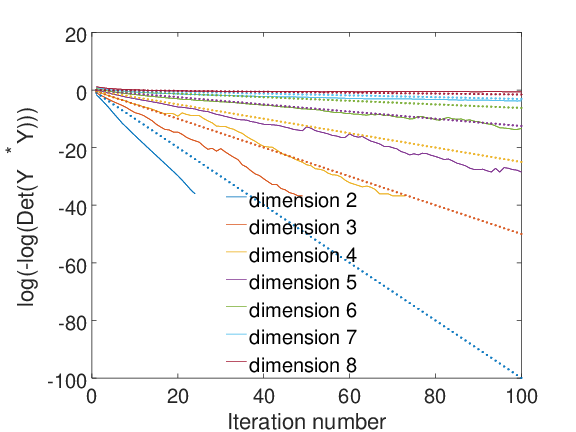}\label{Cr3}}
\caption{Average of $\log(-\log(\det(u^{\dagger}u))$ over 1000 Gaussian 
$\mathcal{N}(0,1)$ random matrices plotted with iteration number. Dotted line 
represent 
$y = -\frac{x}{2^t}$ curves for $t = 0,1,2, \cdots 7$. 
}\label{Cr}
\end{figure}

\section{Proofs: Special cases}
In this section, we consider a matrix of a particular type denoted by $A$. 
Eigenvector matrix of $A$
is denoted by $X_1$. $X_2$ denotes the eigenvector matrix of $X_1$, similarly 
$X_{i}$ denoting
the eigenvector matrix of the matrix $X_{i-1}$. Let all $X_i$ have columns of unit $\mathcal{L}_2$ norm.
We impose condition on choosing the eigenvectors, using the
fact :  if $x$ is an eigenvector of a matrix $A$ satisfying
$Ax = \lambda x$, then $\alpha x$ is also an eigenvector for $\alpha \in 
\mathbb{C}$.
By this conditioning, we force $X_i$ itself to converge to a unitary matrix.

\subsection{Upper triangular matrices}
First, we consider $2 \times 2$ upper triangular matrix of the 
form 
$ A = \begin{bmatrix}
  1 & \cos\theta \\
  0 & \sin\theta \\
 \end{bmatrix},$ when $\theta <0$ (restricted to fourth quadrant).  
 Note that the eigenvectors of the matrix $A$ can be chosen such that 
eigenvector matrix 
 is also of the same form,
 that is, $X_1 =  \begin{bmatrix}
   1&   \cos\beta \\
    0 &  \sin\beta
\end{bmatrix} $
with $\beta <0$ (restricted to fourth quadrant).
\begin{theorem}\label{T1}
For $2 \times 2$ upper triangular matrix of the form
\begin{align}
A = 
\begin{bmatrix}
  1 & \cos\theta_1 \\
  0 & \sin\theta_1 \\
 \end{bmatrix},
\end{align}
if $\theta_1 < 0$ (restricting to fourth quadrant) maintained for all the 
eigenvector matrices $X_i$, then 
$$\lim_{i \to \infty}X_i = \begin{bmatrix}
1 & 0 \\
0 & -1 
\end{bmatrix}.
$$
\end{theorem}
\begin{proof} 
We can see that 
$e1  = 
\begin{bmatrix} 
1\\
 0 
\end{bmatrix}$ 
 is an eigenvector. Let the second eigenvector be
$v = 
\begin{bmatrix}
      \cos\theta_2 \\
      \sin\theta_2
\end{bmatrix}
$.
Then we have 
\begin{align}
& Av = \sin(\theta_1)v \\
& \begin{bmatrix}
  1 & \cos\theta_1 \\
  0 & \sin\theta_1 \\
 \end{bmatrix}
 \begin{bmatrix}
      \cos\theta_2 \\
      \sin\theta_2
\end{bmatrix}
= 
\sin \theta_1
\begin{bmatrix}
      \cos\theta_2 \\
      \sin\theta_2
\end{bmatrix}.
\end{align}
From this relation we get,
\begin{align}
 \cos \theta_2 + \cos \theta_1 \sin \theta_2 = \sin \theta_1 \cos \theta_2 \\
 \cos \theta_2 = \sin (\theta_1 - \theta_2). \label{trig1}
\end{align}
We have $\theta_2$ in the fourth quadrant, so we get 
\begin{align}
 \sin \left( \frac{\pi}{2} + \theta_2 \right ) = \sin (\theta_1 - \theta_2) \\
 \theta_2 = \frac{- \pi }{4} + \frac{\theta_1}{2}. 
\end{align}

When we repeat the procedure, let $\theta$ be the angle corresponding to 
limiting eigenvector, we have
\begin{align}
 \theta &= \frac{ -\pi }{2} \left( \frac{1}{2} + \frac{1}{4} + \frac{1}{8} + 
\frac{1}{16} + \frac{1}{32} \cdots \right)\\
 \theta &= \frac{ -\pi }{2} .
\end{align}
Therefore, in the limit of repeated extrusion of the eigenvectors, the upper 
triangular matrix converges to the diagonal matrix, which is unitary. 
\end{proof}

\begin{theorem}
For a $n \times n$ real upper triangular matrix, \\ 
$A = \begin{bmatrix}
  1 & 0 & 0 & \cdots & x_1 \\
  0 & 1 & 0 & \cdots & x_2 \\
  0 & 0 & \ddots& \vdots& \vdots \\
  0 & 0 &  \cdots& 1 & x_{n-1} \\
  0 & 0 & \cdots & 0 & x_n
 \end{bmatrix}
$  with $\sum_{i=1}^{n} x_i^2 = 1$. and $x_i \geq 0$ for $1 \leq i  < n$ and 
$x_n < 0$. If we impose conditions 
$\sum_{i=1}^{n} y_i^2 = 1$. and $y_i \geq 0$ for $1 \leq i  < n$ and $y_n < 0$
on the eigenvector satisfying,
\begin{align}
   \begin{bmatrix}
  1 & 0 & 0 & \cdots & x_1 \\
  0 & 1 & 0 & \cdots & x_2 \\
  0 & 0 & \ddots& \vdots& \vdots \\
  0 & 0 &  \cdots& 1 & x_{n-1} \\
  0 & 0 & \cdots & 0 & x_n
 \end{bmatrix}
 \begin{bmatrix}
   y_1\\
   y_2\\
  \vdots\\
  y_{n-1}\\
  y_n
\end{bmatrix}
 = 
 x_n 
\begin{bmatrix}
   y_1\\
   y_2\\
  \vdots\\
  y_{n-1}\\
  y_n
\end{bmatrix}.
\end{align}

Then in the process of repeated extrusion of eigenvectors, \\
limiting eigenvector matrix is,
 $ \begin{bmatrix}
  1 & 0 & 0 & \cdots & 0 \\
  0 & 1 & 0 & \cdots & 0 \\
  0 & 0 & \ddots& \vdots& \vdots \\
  0 & 0 &  \cdots& 1 & 0 \\
  0 & 0 & \cdots & 0 & -1
 \end{bmatrix}$. 
\end{theorem}

\begin{proof}
Let
$X_1 = 
\begin{bmatrix}
   1 & 0 & 0 & \cdots & y_1 \\
  0 & 1 & 0 & \cdots & y_2 \\
  0 & 0 & \ddots& \vdots& \vdots \\
  0 & 0 &  \cdots& 1 & y_{n-1} \\
  0 & 0 & \cdots & 0 & y_n
\end{bmatrix}
$ be the eigenvector  matrix of $A$.

So we have the relation, 
\begin{align}
  \begin{bmatrix}
  1 & 0 & 0 & \cdots & x_1 \\
  0 & 1 & 0 & \cdots & x_2 \\
  0 & 0 & \ddots& \vdots& \vdots \\
  0 & 0 &  \cdots& 1 & x_{n-1} \\
  0 & 0 & \cdots & 0 & x_n
 \end{bmatrix}
 \begin{bmatrix}
   y_1\\
   y_2\\
  \vdots\\
  y_{n-1}\\
  y_n
\end{bmatrix}
 = 
 x_n 
\begin{bmatrix}
   y_1\\
   y_2\\
  \vdots\\
  y_{n-1}\\
  y_n
\end{bmatrix}.
\end{align}
This implies ,
\begin{align}
 x_i y_n & = y_i (x_n-1) \text{  for  $i = 1,2, \cdots ,{n-1}$}. \label{eigr1} 
\end{align}
Which means,

\begin{align}
 \frac{x_i}{x_j}& = \frac{y_i}{y_j} \text{  for   $ i, j \in  \{1,2, \cdots 
,{n-1}\}$}. \label{ratio1} 
\end{align}

Let us denote $\cos \theta_i$ by $c_i$ and $\sin \theta_i$ by $s_i$,
then we have eigenvector of  $A$ in the form,
$
\begin{bmatrix}
c_1c_2 \cdots c_n \\
s_1c_2c_3 ...c_n \\
s_2 c_3 \cdots c_n \\
\vdots\\
s_n
\end{bmatrix}$. Because of the constraints we can represent last column of $A$ 
by    
$
\begin{bmatrix}
c_1c_2 \cdots c_n \\
s_1c_2c_3 ...c_n \\
s_2 c_3 \cdots c_n \\
\vdots\\
s_n
\end{bmatrix}$ and eigenvector by $
\begin{bmatrix}
c_{n+1}c_{n+2} \cdots c_{2n} \\
s_{n+1}c_{n+2}c_{n+3} \cdots c_{2n} \\
s_{n+2} c_{n+3} \cdots c_{2n} \\
\vdots\\
s_{2n}
\end{bmatrix}$. Then from the relation \eqref{ratio1} and the positivity of 
$x_i$ and $y_i$, we have $ \theta_i = \theta_{i+n}$ for $1 \leq i <n$.

From the equation \eqref{eigr1} we get,
\begin{align}
 s_{2n-1}c_{2n} +  s_{n-1} c_n s_{2n} &= s_n s_{2n-1} c_{2n} \\
 c_{2n} &= s_n c_{2n} - c_ns_{2n} \label{trig3}
\end{align}

Note that equation \eqref{trig3} is same as equation \eqref{trig1}, so from the 
arguments as in theorem \ref{T1},
we get the limiting eigenvector matrix,
 $ \begin{bmatrix}
  1 & 0 & 0 & \cdots & 0 \\
  0 & 1 & 0 & \cdots & 0 \\
  0 & 0 & \ddots& \vdots& \vdots \\
  0 & 0 &  \cdots& 1 & 0 \\
  0 & 0 & \cdots & 0 & -1
 \end{bmatrix}$.

\end{proof}

\subsubsection{Loops and discontinuities}\label{loop}
While selecting the condition on the eigenvectors to make $X_i$ itself converge to an
orthogonal matrix, we may get into two difficult scenarios.
\begin{itemize}
\item In general, for upper triangular matrices, if we don't have the restriction on 
the eigenvectors, the procedure of repeated extrusion of eigenvectors
may get into a loop. For example, the matrices 
$
\begin{bmatrix}
1 & \frac{\sqrt{3}}{2} \\
0 &\frac{1}{2} 
\end{bmatrix}
$ and the matrix  
$\begin{bmatrix}
1 & \frac{-\sqrt{3}}{2} \\
0 & \frac{1}{2} 
\end{bmatrix}
$ are the eigenvector matrices of each other.\\
\item Discontinuity: Consider a $2 \times 2$ upper triangular matrix 
$\begin{bmatrix}
  1 & \epsilon \\
  0 & \sqrt{1-\epsilon^2}
 \end{bmatrix}
$ for a small $\epsilon >0$ we can see that its eigenvector  $\begin{bmatrix}
                                                               x \\
                                                               y
                                                              \end{bmatrix}
$
corresponding to eigenvalue $\sqrt{1-\epsilon^2}$, then as $\epsilon \to 0$, we 
have $y \to 0$ and $x \to 1$. So we have the following

 $\begin{bmatrix}
  1 & \epsilon \\
  0 & \sqrt{1-\epsilon^2}
 \end{bmatrix}
$ matrix has eigenvector matrix 
 $\begin{bmatrix}
  1 & \pm \sqrt{1-\delta^2} \\
  0 & \delta
 \end{bmatrix}
$ for small $\epsilon ,\delta > 0$.
\end{itemize}
\subsubsection{$3 \times 3 $ upper triangular matrices}
\begin{theorem}\label{T3}
 For a $3 \times 3$ upper triangular matrix of the form $A = \begin{bmatrix}
                                                              1 & 0 & a \\
                                                              0 & -1 & b \\
                                                              0 & 0 & ci 
                                                             \end{bmatrix}
$ where $|a|^2 + |b|^2 + |c|^2 = 1$ and $c \in \mathbb{R},$ and $a,b \in 
\mathbb{C}$  if we maintain the 
constraint to the eigenvector matrix
$X_1 = \begin{bmatrix}
                                                              1 & 0 & x \\
                                                              0 & -1 & y \\
                                                              0 & 0 & zi 
                                                             \end{bmatrix},
$ $|x|^2 + |y|^2 + |z|^2 = 1$ and $z \in \mathbb{R},$ and $x,y \in \mathbb{C}$ 
in the limit of repeatedly taking eigenvectors, 
we get the matrix
$\begin{bmatrix}
                                                              1 & 0 & 0 \\
                                                              0 & -1 & 0 \\
                                                              0 & 0 & i 
                                                             \end{bmatrix}$.

\end{theorem}
\begin{proof}
 We have a relation with eigenvectors,
 \begin{align}
 A = \begin{bmatrix}
1 & 0 & a \\
0 & -1 & b \\
0 & 0 & ci 
\end{bmatrix}   
\begin{bmatrix}
 x \\
 y \\
zi 
\end{bmatrix}
= 
ci \begin{bmatrix}
 x \\
 y \\
zi 
\end{bmatrix}.
\end{align}
From this relation we get,

\begin{align}
 x + azi &= cix \\
 -y + bzi &= ci y.
\end{align}
Rearranging,
\begin{align}
 azi &= x (ci-1), \label{er1}\\
 bzi &= y (ci+1)\label{er2}.
\end{align}
Using the ratio of equations \eqref{er1} and \eqref{er2},

\begin{align}
 \frac{a}{b} &= \frac{x}{y}\frac{ci-1}{ci+1} \label{t1r}
\end{align}

Note $ -1i \leq ci \leq 1i $ so $\frac{ci-1}{ci+1}$ is a point on the unit 
circle, so using the magnitude in \eqref{t1r},
we get
\begin{align}
 \frac{|a|}{|b|} &= \frac{|x|}{|y|} \label{t2r}
\end{align}

If we represent vector $\begin{bmatrix} 
a \\
b\\
ci
\end{bmatrix}$ as 
$\begin{bmatrix} 
\cos \theta_1 \cos \theta_2 e^{i \alpha} \\
\sin \theta_1 \cos \theta_2 e^{i \beta}\\
\sin \theta_2 i
\end{bmatrix}$, and vector $\begin{bmatrix}
 x \\
 y \\
zi 
\end{bmatrix}$
as 
$\begin{bmatrix} 
\cos \theta_3 \cos \theta_4 e^{i \gamma} \\
\sin \theta_3 \cos \theta_4 e^{i \delta}\\
\sin \theta_4 i
\end{bmatrix}$.

Then we have from the relation \eqref{t2r}, $\tan \theta_1 =\tan \theta_3$. Thus 
we have $\theta_1 = \theta_3$.

Multiplying the equations  \eqref{er1} and \eqref{er2}, the angles are given by
$\alpha+\beta = \gamma + \delta$. Also using the magnitude and the fact 
$\theta_1 = \theta_3$, we get

\begin{align}
 \cos^2 \theta_2 \sin^2 \theta_4 &= \cos^2 \theta_4 (1 + \sin^2 \theta_2) \\
 \tan^2 \theta_4 &= 1 + 2 \tan^2 \theta_2 \label{limitr}
\end{align}

In the limit of repeated extrusion of eigenvectors \eqref{limitr} gives 
\begin{align}
 tan^2 \theta = 1+ 2+ 4+ 8 + 16  \cdots 
\end{align}

So the limiting eigenvector matrix is,
$\begin{bmatrix}
1 & 0 & 0 \\
 0 & -1 & 0 \\
0 & 0 & i 
\end{bmatrix}$.
\end{proof}

\begin{corollary}
 For a real upper triangular matrix of dimension 3, while recursive evaluation 
eigenvector matrices, first constraining the second eigenvector as in theorem 
\ref{T1} we obtain $\begin{bmatrix}
          1 & \epsilon & a \\
          0 & -1 + \delta & b \\
          0 & 0 & c 
         \end{bmatrix}
$ for small $\epsilon,\delta >0$ ,  then constraining the third column as in 
theorem \ref{T3} 
along with the second column, gives  
 $\begin{bmatrix}
          1 & 0 & 0 \\
          0 & -1  &0\\
          0 & 0 & i
         \end{bmatrix} $ as  a limiting eigenvector matrix.
\end{corollary}

\subsection{A special $2 \times 2$ matrix}
Consider a $2 \times 2$ matrix with real entries such that 
$A = \begin{bmatrix}
      \cos \theta_1 & \cos \theta_2 \\
      \sin \theta_1 & \sin \theta_2
      \end{bmatrix}
$
with $\theta_1$ restricted to the first quadrant and
$\theta_1 < \frac{\pi}{4}$ and $\theta_2 $ restricted to the fourth quadrant and 
$|\theta_2| > \frac{\pi}{4}$.
Then we can form an eigenvector matrix of the 
same form $X = \begin{bmatrix}
      \cos \theta_3 & \cos \theta_4 \\
      \sin \theta_3 & \sin \theta_4
      \end{bmatrix}
$ with the restrictions $\theta_3$, $\theta_4$ respectively in the first and 
fourth quadrants. If $\theta_1-\frac{\pi}{2} = \theta_2$, then the matrix is 
orthogonal, and so we exclude that condition. We prove that in the limit of 
repeatedly taking eigenvectors, the eigenvector matrix converges to 
$\begin{bmatrix}
1 & 0 \\
0 & -1 
\end{bmatrix}.
$

\begin{theorem}
 For $2 \times 2$ real matrix of the form
\begin{align}
A = 
\begin{bmatrix}
      \cos \theta_1 & \cos \theta_2 \\
      \sin \theta_1 & \sin \theta_2
      \end{bmatrix},
\end{align} starting with $\theta_1 < \frac{\pi}{4}, |\theta_2| > \frac{\pi}{4} 
$ and $\theta_1 - \frac{\pi}{2} \neq \theta_2$,
if $\theta_1$ restricted to the first quadrant and $\theta_2$ restricted to the 
fourth quadrant for all the eigenvector matrices $X_i$, then 
$$\lim_{i \to \infty}X_i = \begin{bmatrix}
1 & 0 \\
0 & -1 
\end{bmatrix}.
$$
\end{theorem}
\begin{proof}
 Consider the first eigenvector matrix, $X = \begin{bmatrix}
      \cos \theta_3 & \cos \theta_4 \\
      \sin \theta_3 & \sin \theta_4
      \end{bmatrix}$  We have from the Gerschgorin circle theorem the 
eigenvalues $\lambda_1$ and $\lambda_2$ of matrix $A$ satisfying,
  \begin{align}
   \cos \theta_1  - \sin \theta_1 & \leq \lambda_1 \leq   \cos \theta_1  + \sin 
\theta_1  \\
   \sin \theta_2 - \cos \theta_2 & \leq \lambda_2 \leq  \sin \theta_2 + \cos 
\theta_2.
  \end{align}
 We have 
 \begin{align}
  \begin{bmatrix}
      \cos \theta_1 & \cos \theta_2 \\
      \sin \theta_1 & \sin \theta_2
      \end{bmatrix} 
      \begin{bmatrix}
      \cos \theta_3  \\
      \sin \theta_3 
      \end{bmatrix} = \lambda_1  \begin{bmatrix}
      \cos \theta_3  \\
      \sin \theta_3 
      \end{bmatrix}.
 \end{align}
 
 Which gives 
 \begin{align}
  \sin \theta_1  \cos \theta_3   +  \sin \theta_2 \sin \theta_3 & = \lambda_1  
\sin \theta_3  \\
  \sin \theta_1  \cos \theta_3   +  \sin \theta_2 \sin \theta_3 & \geq  (\cos 
\theta_1  - \sin \theta_1)   \sin \theta_3.  \label{gr1}
\end{align}  
Using the fact that $\sin \theta_2 < 0$, and $- \sin \theta_1 - \sin \theta_2  
>0$,
\begin{align}
  \frac{\sin \theta_1}{\cos \theta_1  - \sin \theta_1 - \sin \theta_2 } &\geq  
\frac{\sin \theta_3}{ \cos \theta_3} \label{t3relation}\\
  \tan \theta_1 & >  \tan \theta_3  
 \end{align}
This implies $\theta_1 > \theta_3$. 

Similarly we have,

 \begin{align}
  \begin{bmatrix}
      \cos \theta_1 & \cos \theta_2 \\
      \sin \theta_1 & \sin \theta_2
      \end{bmatrix} 
      \begin{bmatrix}
      \cos \theta_4  \\
      \sin \theta_4 
      \end{bmatrix} = \lambda_2  \begin{bmatrix}
      \cos \theta_4  \\
      \sin \theta_4 
      \end{bmatrix}.
 \end{align}
 
 Which gives, 
 \begin{align}
   \cos \theta_1 \cos \theta_4  +  \cos \theta_2 \sin \theta_4 &= \lambda_2  
\cos \theta_4  \\
    \cos \theta_1 \cos \theta_4  +  \cos \theta_2 \sin \theta_4 & \leq (\sin 
\theta_2 + \cos \theta_2 )  \cos \theta_4.  \label{gr2}
 \end{align}

 Rearranging \eqref{gr2} gives
 
 \begin{align}
  \frac{\sin \theta_4}{\cos \theta_4} & \leq \frac{\sin \theta_2 + \cos \theta_2 
-\cos \theta_1}{\cos \theta_2} \\
  \frac{\sin (-\theta_2) - \cos \theta_2 + \cos \theta_1}{\cos \theta_2}  &\leq 
\frac{\sin (-\theta_4)}{\cos \theta_4} \label{t4relation} \\
  \tan |\theta_2| < \tan |\theta_4|.
 \end{align}
This implies $|\theta_2| < |\theta_4|$. So the matrices formed are in a 
increasing sequence of angles corresponding to $|\theta_2|$ and decreasing
sequence of angles corresponding to $\theta_1$.

Let us denote $-\sin \theta_2 - \sin \theta_1$ by $\delta >0$, from equation 
\eqref{t3relation} we have,

\begin{align}
 \tan \theta_1 \frac{1}{1 + \frac{\delta}{\cos \theta_1}} \geq \tan \theta_3  \\
 \tan \theta_1 \left( 1- \frac{\delta}{\cos \theta_1} + \left(\frac{\delta}{\cos 
\theta_1} \right) ^2 \right) > \tan \theta_3 \\
 \tan \theta_1 - \tan \theta_3   > \tan \theta_1  \left( \frac{\delta}{\cos 
\theta_1} - \left(\frac{\delta}{\cos \theta_1} \right) ^2 \right). \label{t3lb}
\end{align}
From equation \eqref{t3lb} we can see that, for sufficiently large $\theta_1 > 0$, difference 
between $\tan \theta_3$ and $\tan \theta_1$ are not arbitrary small. So the 
sequence corresponding to $\theta_1$ should converge to zero.
Similarly we get the relation from equation \eqref{t4relation},

\begin{align}
 \frac{\cos \theta_1 - \cos \theta_2}{\cos \theta_2} \leq \tan |\theta_4| - \tan 
|\theta_2| .
\end{align}
Which says that consecutive difference between $\tan \theta$ increases and hence 
the sequence correspond to angles will converge to $ \frac{-\pi}{2}$.

So in the limit of repeatedly taking eigenvectors, the eigenvector matrix 
converges to $\begin{bmatrix}
1 & 0 \\
0 & -1 
\end{bmatrix}$.

\end{proof}

\section*{Conclusion}
The proposed conjecture is on the limiting behavior of recursive evaluation of 
eigen directions of linear operators up to dimension seven. It is to be proven in general that, the limiting
eigenvector matrix having columns as points on the unit $n$-ball, tend to be unitary.
Numerical results clearly show the convergence behavior up-to dimension seven, supporting the conjecture.
Some special cases of this recursive eigen extrusion procedure are proved affirmatively. 
Thus opening the further connection between the geometry of unit n-ball and eigen properties of a linear operator. 
\clearpage
\subsection*{Acknowledgement}
The author would like to thank Prof. Murugesan Venkatapathi  for much helpful discussions and insights.

\vspace{2mm}
\bibliographystyle{siam}
\bibliography{references1}
\end{document}